\documentclass[11pt,reqno]{amsart}
\usepackage[a4paper,margin=27mm]{geometry}
\usepackage[T1]{fontenc}
\usepackage{lmodern,amsmath,amssymb,amsthm,mathtools,microtype}
\usepackage{xcolor}
\definecolor{linkblue}{RGB}{25,58,108}
\usepackage[colorlinks=true,linkcolor=linkblue,citecolor=linkblue,urlcolor=linkblue]{hyperref}
\newtheorem{theorem}{Theorem}[section]
\newtheorem{lemma}[theorem]{Lemma}

\newtheorem{corollary}[theorem]{Corollary}
\theoremstyle{remark}
\newcommand{\E}{\mathbb E}
\newcommand{\R}{\mathbb R}
\newcommand{\dd}{\,\mathrm d}
\newcommand{\CJ}{\mathrm{CJ}}
\newcommand{\HP}{\mathrm{HP}}
\newcommand{\Sine}{\mathsf{Sine}}

\title[Critical and subcritical fusion for Sine beta]{Critical and subcritical fusion asymptotics\\for $\Sine_\beta$ correlation functions}
\author{Weiyang Fang}
\address{Independent Researcher, Tokyo, Japan}
\email{weiyang.fang@hotmail.com}
\date{September 7, 2026}
\begin{document}
\begin{abstract}
We determine the first correction to the leading Vandermonde fusion law for the correlation functions of the $\Sine_\beta$ process in the critical and subcritical regimes $m\beta\leq1$. When $m\beta<1$, the normalized correction is of order $|\varepsilon|^{1+m\beta}$, with a strictly negative coefficient given by an explicit gamma-function ratio times an absolutely convergent arithmetic--geometric mean deficit integral. At $m\beta=1$, it is
\[
 -\frac{\sum_{i<j}(a_i-a_j)^2}{8m^2(2m+1)}
      \varepsilon^2\log(1/|\varepsilon|)+O(\varepsilon^2).
\]
For two merging points, the subcritical coefficient reduces to a gamma-function expression involving $\sec(\pi\beta)-1$, and the critical logarithmic coefficient is $-1/160$. The argument starts from a geometric interpolation of circular-Jacobi weights. Selecting one particle in the interpolation derivative increases the fused charge by $\beta$ and leaves a strict positive-power moment margin in the remaining stochastic-zeta expectation. This yields an absolutely convergent one-particle identity and uniform control at the collision scale. The results resolve the critical and subcritical conjecture in the author's earlier preprint and complement its supercritical second-order expansion.
\end{abstract}
\maketitle

\section{Introduction and main results}\label{sec:intro}
The $\Sine_\beta$ process is the translation-invariant bulk limit of one-dimensional beta ensembles. Its Brownian-carousel and random-operator descriptions were developed by Valk\'o and Vir\'ag \cite{VV09,VV17}; the circular-Jacobi limits of Li and Valk\'o \cite{LV} provide the corresponding Hua--Pickrell environments and normalized characteristic polynomials. We use the normalization in which the intensity of $\Sine_\beta$ is $1/(2\pi)$.

For general $\beta>0$, Qu and Valk\'o \cite{QV} obtained a stochastic representation of the pair correlation. Assiotis and Najnudel \cite{AN} subsequently expressed every correlation function in terms of the Hua--Pickrell stochastic zeta function and obtained its leading behavior when several arguments merge. Specifically, their formula is
\begin{equation}\label{eq:ANcorrelation}
 \rho_\beta^{(m)}(x_1,\ldots,x_m)
 =C_\beta^{(m)}\prod_{i<j}|x_i-x_j|^\beta
       \E\prod_{j=2}^m|\xi^{\beta,m\beta/2}(x_j-x_1)|^\beta.
\end{equation}
The normalized expectation tends to one at a full collision. The next term is sensitive to the integrability of the fused environment near the collision point.

In the earlier preprint \cite{Fang}, the author computed this next term in the supercritical range $m\beta>1$ and proposed a transition at $m\beta=1$. The supercritical coefficient has a simple pole at the threshold, where the inverse-square moment used in that argument ceases to be finite. The present paper treats the two complementary regimes. Its new contributions are the exact subcritical coefficient, the critical logarithm with an $O(\varepsilon^2)$ remainder, and a finite-dimensional interpolation identity that supplies the necessary localization directly. The estimates used below are derived from the circular-Jacobi model and the stochastic-zeta bounds of \cite{AN}; the supercritical theorem of \cite{Fang} is used only for comparison.

Fix $m\geq2$, $\beta>0$, and a pairwise distinct real profile $a=(a_1,\ldots,a_m)$. Write
\begin{equation}\label{eq:geometry}
 \alpha=m\beta,\qquad \bar a=\frac1m\sum_i a_i,\qquad
 b_i=a_i-\bar a,\qquad
 B=\sum_i b_i^2=\frac{V(a)}m,
\end{equation}
where $V(a)=\sum_{i<j}(a_i-a_j)^2$. Define
\begin{equation}\label{eq:R}
 R_{\beta,m}(\varepsilon;a)=
 \frac{\rho_\beta^{(m)}(\varepsilon a_1,\ldots,\varepsilon a_m)}
 {C_\beta^{(m)}|\varepsilon|^{\beta\binom m2}\prod_{i<j}|a_i-a_j|^\beta}.
\end{equation}
Translation and reflection invariance allow us to center the profile and take $\varepsilon>0$ in the proof. The constant that will enter the answer is
\begin{equation}\label{eq:kappa}
 \kappa_{\beta,m}=
 \frac{(\beta/2)^{m\beta}\Gamma(1+\beta/2)\Gamma(1+m\beta/2)^2}
 {2\pi\Gamma(1+m\beta)\Gamma(1+(m+\tfrac12)\beta)}
 =\frac{C_\beta^{(m+1)}}{C_\beta^{(m)}}.
\end{equation}
The normalization is derived in Section~\ref{sec:normalization}.
\begin{theorem}[Critical and subcritical fusion]\label{thm:main}
If $0<\alpha<1$, set
\begin{equation}\label{eq:J}
 \mathcal J_{\beta,m}(a)=\int_\R
 \left\{\frac1m\sum_{i=1}^m |u-a_i|^{m\beta}
             -\prod_{i=1}^m|u-a_i|^\beta\right\}\dd u.
\end{equation}
The integral is absolutely convergent and strictly positive. As $\varepsilon\to0$,
\begin{equation}\label{eq:sub}
 R_{\beta,m}(\varepsilon;a)
 =1-\kappa_{\beta,m}\mathcal J_{\beta,m}(a)|\varepsilon|^{1+m\beta}
     +o(|\varepsilon|^{1+m\beta}).
\end{equation}
If $\alpha=1$, equivalently $\beta=1/m$, then
\begin{equation}\label{eq:crit}
 R_{1/m,m}(\varepsilon;a)
 =1-\frac{V(a)}{8m^2(2m+1)}\varepsilon^2\log\frac1{|\varepsilon|}
     +O(\varepsilon^2).
\end{equation}
The remainders are locally uniform in the collision profile. The integral in \eqref{eq:J} is translation invariant and homogeneous of degree $1+m\beta$ under nonzero real scaling of the profile.
\end{theorem}

\begin{corollary}[Two-point fusion]\label{cor:pair}
For $0<\beta<1/2$, put
\begin{equation}\label{eq:D}
 D_\beta=\kappa_{\beta,2}
 \frac{\Gamma(1+\beta)^2}{\Gamma(2+2\beta)}
                    \bigl(\sec(\pi\beta)-1\bigr)>0.
\end{equation}
Then
\begin{equation}\label{eq:pairsub}
 \rho_\beta^{(2)}(0,x)
 =C_\beta^{(2)}|x|^\beta
       \left(1-D_\beta|x|^{1+2\beta}+o(|x|^{1+2\beta})\right).
\end{equation}
At $\beta=1/2$,
\begin{equation}\label{eq:paircrit}
 \rho_{1/2}^{(2)}(0,x)
 =C_{1/2}^{(2)}|x|^{1/2}
       \left(1-\frac{x^2}{160}\log\frac1{|x|}+O(x^2)\right).
\end{equation}
\end{corollary}

The first correction is strictly negative for every distinct profile. Below the threshold, the shape dependence is encoded by the nonnegative integral $\mathcal J_{\beta,m}$, whose degree of homogeneity is $1+m\beta$. At criticality, the coefficient depends only on the quadratic discriminant $V(a)$. For reference, the complementary result in \cite[Theorem 1.1]{Fang} is
\begin{equation}\label{eq:supercriticalbackground}
 R_{\beta,m}(\varepsilon;a)
 =1-\frac{\beta^2V(a)}{8(m\beta-1)(2m+1)}\varepsilon^2
      +o(\varepsilon^2),\qquad m\beta>1.
\end{equation}
Thus the three correction scales are $|\varepsilon|^{1+m\beta}$, $\varepsilon^2\log(1/|\varepsilon|)$, and $\varepsilon^2$, respectively.

The proof interpolates between a fused charge of size $m\beta$ at the origin and $m$ charges of size $\beta$ at the displaced points. Differentiating the interpolation parameter produces a sum over environmental particles. Once one such particle is selected, its interaction with the others adds a charge $\beta$ at its location. The remaining normalized-polynomial factors have total exponent $m\beta$, while the circular-Jacobi tilt is $(m+1)\beta$. This strict gap controls the limiting expectation even when the inverse-square moment in the original fused environment diverges. Centering the profile removes the first-order far-field term, leaving an integrable tail after the stochastic-zeta moment bound is applied.

Section~\ref{sec:inputs} records the analytic inputs and their compact-moment consequence. Section~\ref{sec:interpolation} proves the interpolation identity. Sections~\ref{sec:subcritical} and \ref{sec:critical} establish the two asymptotic regimes. Section~\ref{sec:normalization} computes the constants, Section~\ref{sec:pair} evaluates the pair coefficient, and Section~\ref{sec:discussion} discusses the relation with the supercritical pole and further questions.

\section{Circular-Jacobi and stochastic-zeta inputs}\label{sec:inputs}
Let $\mathbb T=[-\pi,\pi)$, $d\mu(\theta)=d\theta/(2\pi)$, and define
\begin{equation}\label{eq:ZN}
 Z_N(q)=\int_{\mathbb T^N}\prod_{j<k}|e^{i\theta_j}-e^{i\theta_k}|^\beta
       \prod_{j=1}^N|1-e^{i\theta_j}|^q\prod_jd\mu(\theta_j),\qquad q\geq0.
\end{equation}
We put $Z_0(q)=1$. Dividing the integrand by $Z_N(2\delta)$ gives the law $\CJ_{N,\beta,\delta}$. Under this law, set
\begin{equation}\label{eq:normalizedpolynomial}
 q_N(z)=\prod_{j=1}^N\frac{z-e^{i\theta_j}}{1-e^{i\theta_j}},
 \qquad f_N(z)=e^{-iz/2}q_N(e^{iz/N}).
\end{equation}
For real $t$, $f_N(t)$ is real and $f_N(0)=1$.

The following facts are the external probabilistic inputs. Under the coupling of \cite[Propositions 2.7--2.8]{AN}, originating in \cite{LV}, $f_N$ converges almost surely, locally uniformly on $\mathbb C$, to $\xi^{\beta,\delta}$ with $\xi^{\beta,\delta}(0)=1$. Its real zeros form $\HP_{\beta,\delta}$. For fixed $0<r\leq2\delta$, the one-point case of \cite[Theorem 2.11]{AN}, on a fundamental interval, gives
\begin{equation}\label{eq:realbound}
 \E|q_N(e^{iy/N})|^r
 \leq\frac{C_{\beta,\delta,r}}{1+|y|^{(2\delta-r)r/\beta}},
 \qquad |y|\leq\pi N.
\end{equation}
For other real $y$, the finite-$N$ bound is interpreted using its representative modulo $2\pi N$. Fatou's lemma gives the corresponding real-line bound for $\xi^{\beta,\delta}$. We also use the exact correlation formula \eqref{eq:ANcorrelation}, proved in \cite[Theorem 1.8]{AN}.

\begin{lemma}[Compact moments]\label{lem:compact}
For $\beta,\delta>0$, $0<r<2\delta$, $R>0$, and every integer $k\geq0$,
\begin{equation}\label{eq:compact}
 \sup_{N\geq1}\E\sup_{|z|\leq R}|f_N^{(k)}(z)|^r<\infty,
 \qquad \E\sup_{|z|\leq R}|(\xi^{\beta,\delta})^{(k)}(z)|^r<\infty.
\end{equation}
\end{lemma}
\begin{proof}
For $0<r\leq2\delta$, \eqref{eq:realbound} implies a uniform bound on the circle. Subharmonicity of $|q_N|^r$ and the Poisson inequality extend this bound to the unit disk. Because all zeros of $q_N$ are on the unit circle,
\[
 |q_N(z)|=|z|^N|q_N(1/\overline z)|,\qquad |z|>1.
\]
The normalization in \eqref{eq:normalizedpolynomial} therefore gives
\[
 \E|f_N(x+iy)|^r\leq C e^{r|y|/2}.
\]
This is the complex-plane estimate used in \cite[Proposition 2.10, equations (37)--(38)]{AN}. Fatou's lemma gives it for the limiting entire function as well.

For an entire function $g$, $|g|^r$ is subharmonic. The Poisson inequality on a disk of radius $R+1$ bounds its supremum on $|z|\leq R$ by a constant times the boundary integral on the larger circle. Taking expectations yields \eqref{eq:compact} for $k=0$. Cauchy's estimate on concentric disks gives the assertion for every fixed $k$.
\end{proof}

\begin{lemma}[Rotational normalization]\label{lem:rotation}
For $q>0$ and every real $u$,
\begin{equation}\label{eq:rotation}
 \E_{\CJ_{N,\beta,q/2}}|q_N(e^{iu})|^q=1.
\end{equation}
Consequently, if $r_j\geq0$ and $\sum_jr_j\leq q$, then
\[
 \E_{\CJ_{N,\beta,q/2}}\prod_j|q_N(e^{iu_j})|^{r_j}\leq1.
\]
\end{lemma}
\begin{proof}
Multiplication by $|q_N(e^{iu})|^q$ replaces the weight at $1$ by the same weight at $e^{iu}$. Rotational invariance of the unweighted circular ensemble proves \eqref{eq:rotation}. Generalized H\"older, adding a constant factor when the sum of the exponents is smaller than $q$, proves the second statement.
\end{proof}
\section{A one-particle interpolation identity}\label{sec:interpolation}
Write $\alpha=m\beta$ and use the centered profile $b$ from \eqref{eq:geometry}. Define
\begin{gather*}
 P_t(x)=\prod_{i=1}^m|x-tb_i|^\beta,\qquad
 W_{s,t}(x)=|x|^{\alpha(1-s)}P_t(x)^s,\\
 L_t(x)=\log\frac{P_t(x)}{|x|^\alpha},\qquad 0\leq s\leq1.
\end{gather*}
Values at the finitely many zeros may be assigned arbitrarily when these quantities appear under an integral. Let
\[
 \zeta=\xi^{\beta,(\alpha+\beta)/2}
\]
and define the nonnegative interpolation kernel
\begin{equation}\label{eq:K}
 K_{s,t}(x)=\E\left[
 |\zeta(-x)|^{\alpha(1-s)}
 \prod_{i=1}^m|\zeta(tb_i-x)|^{\beta s}\right].
\end{equation}
Factors with exponent zero are omitted.

\begin{lemma}[Interpolation identity and kernel bounds]\label{lem:interpolation}
For $0<\alpha\leq1$ and sufficiently small $t>0$,
\begin{equation}\label{eq:identity}
 R_{\beta,m}(t;b)-1
 =\kappa_{\beta,m}\int_\R\int_0^1
       W_{s,t}(x)L_t(x)K_{s,t}(x)\dd s\dd x.
\end{equation}
The double integral is absolutely convergent. Uniformly over bounded centered profiles and sufficiently small $t$,
\begin{equation}\label{eq:Kbounds}
 0\leq K_{s,t}(x)\leq1,
 \qquad K_{s,t}(x)\leq\frac{C}{1+|x|^\alpha}.
\end{equation}
Also $K_{s,t}(x)\to1$ as $(t,x)\to(0,0)$, uniformly in $s\in[0,1]$ and bounded profiles. At $\alpha=1$, this improves to
\begin{equation}\label{eq:Klipschitz}
 |K_{s,t}(x)-1|\leq C(|x|+t),\qquad |x|+t\leq r_0.
\end{equation}
\end{lemma}

\subsection{Finite-dimensional derivation}
Use the partition function $Z_N(q)$ from \eqref{eq:ZN}.
Define a partition-function ratio using $N$ environmental particles:
\begin{align*}
 R_{N,s}(t)&=\frac1{Z_N(\alpha)}
  \int_{\mathbb T^N}\prod_{j<k}|e^{i\theta_j}-e^{i\theta_k}|^\beta
                         \prod_{j=1}^N w_{N,s,t}(\theta_j)\prod_jd\mu(\theta_j),\\
 w_{N,s,t}(\theta)&=|1-e^{i\theta}|^{\alpha(1-s)}
                       \prod_i|e^{itb_i/N}-e^{i\theta}|^{\beta s}.
\end{align*}
Then $R_{N,0}=1$. Differentiation with respect to $s$ introduces only logarithmic factors. For fixed $N,t$, all powers in $w_{N,s,t}$ are nonnegative; the Vandermonde is bounded and logarithmic singularities on the circle are integrable. Differentiation and integration in $s$ are therefore justified, including the one-sided endpoint derivatives.

Put $d_N(x)=2N|\sin(x/(2N))|$ and
\[
 P_{N,t}(x)=\prod_i d_N(x-tb_i)^\beta,\quad
 W_{N,s,t}(x)=d_N(x)^{\alpha(1-s)}P_{N,t}(x)^s,\quad
 L_{N,t}(x)=\log\frac{P_{N,t}(x)}{d_N(x)^\alpha}.
\]
Choose one particle in the derivative, write its coordinate as $x/N$, and rotate the other $N-1$ particles by $-x/N$. The interaction with the chosen particle supplies an additional power $\beta$ at its location. Factor out a fused weight of total exponent $\alpha+\beta$ at that location, and use $\CJ_{N-1,\beta,(\alpha+\beta)/2}$ as the reference measure for the remaining particles. The ratios of the displaced weights to the fused weight are normalized-polynomial factors. Consequently,
\begin{equation}\label{eq:finiteid}
 R_{N,1}(t)-1=\kappa_N\int_{-\pi N}^{\pi N}\int_0^1
 W_{N,s,t}(x)L_{N,t}(x)K_{N,s,t}(x)\dd s\dd x,
\end{equation}
where
\begin{equation}\label{eq:kappaNdef}
 \kappa_N=\frac{Z_{N-1}(\alpha+\beta)}{2\pi N^\alpha Z_N(\alpha)}
\end{equation}
and, for the normalized circular-Jacobi polynomial $q_{N-1}$,
\[
 K_{N,s,t}(x)=\E\left[
 |q_{N-1}(e^{-ix/N})|^{\alpha(1-s)}
 \prod_i|q_{N-1}(e^{i(tb_i-x)/N})|^{\beta s}\right].
\]
For $N=1$, the empty polynomial is one. The factor $N$ from choosing a particle cancels the $1/N$ in $d\theta=dx/N$; the remaining $N^{-\alpha}$ comes from the weight of that particle. This explains every scaling factor in \eqref{eq:kappaNdef}.

\subsection{Bounds and passage to the Sine limit}
The total exponent in $K_{N,s,t}$ is $\alpha$, whereas the circular-Jacobi tilt is $\alpha+\beta$. Lemma~\ref{lem:rotation} gives, exactly at finite $n$,
\[
 \E_{\CJ_{n,\beta,q/2}}|q_n(e^{iu})|^q=1.
\]
Generalized H\"older, with $q=\alpha+\beta$ and an extra constant factor, therefore gives $K_{N,s,t}\leq1$.

Equation~\eqref{eq:realbound} gives, for $n=N-1$,
\[
 \E_{\CJ_{n,\beta,(\alpha+\beta)/2}}
       |q_n(e^{iy/N})|^\alpha
 \leq\frac{C}{1+\operatorname{dist}(y,2\pi N\mathbb Z)^\alpha}.
\]
Here and below the bound is used on circular representatives; scaling by $n/N$ only changes constants for $N\geq2$. Apply H\"older with the individual exponents divided by $\alpha$. For $x\in[-\pi N,\pi N]$ and bounded $tb_i$, the circular distances of $-x$ and $tb_i-x$ are comparable to $|x|$ away from a bounded interval. This proves the finite-$N$ analogue of the second estimate in \eqref{eq:Kbounds}, uniformly in $s$. It uses a single fixed moment exponent, so no uniformity of theorem constants in varying exponents is needed.

By Section~\ref{sec:inputs}, the normalized entire functions admit a locally uniform coupling limit. For any
\begin{equation}\label{eq:margin}
 \alpha<q<\alpha+\beta,
\end{equation}
their compact suprema have bounded $q$th moments by Lemma~\ref{lem:compact}, applied with $2\delta=\alpha+\beta$. Since the total exponent in the kernel is $\alpha$, \eqref{eq:margin} gives uniform integrability. Thus $K_{N,s,t}(x)\to K_{s,t}(x)$ for fixed $s,t,x$, and proves the bounds and local continuity asserted in \eqref{eq:Kbounds}. The same compact majorant gives uniformity in $s$ and bounded profiles near $(t,x)=(0,0)$, where the limiting entire function equals one.

To justify convergence of the integrals, let $M\geq1$ bound $\max_i|b_i|$. Since $\sum_i b_i=0$, Taylor's formula for $\log d_N$ gives, when $|x|>2Mt$,
\begin{equation}\label{eq:tailfinite}
 |L_{N,t}(x)|\leq C\frac{t^2}{x^2},\qquad
 W_{N,s,t}(x)\leq C|x|^\alpha.
\end{equation}
Indeed $(\log d_N)''(x)=-[4N^2\sin^2(x/(2N))]^{-1}$ away from its zeros. Combining \eqref{eq:tailfinite} with the kernel bound gives an integrable $Ct^2/x^2$ majorant outside a fixed interval. On a fixed bounded interval, the weights are bounded and logarithmic singularities at $0,tb_1,\ldots,tb_m$ have an integrable majorant, uniform for sufficiently large $N$. Dominated convergence applies for every fixed $t$.

Finally, rotate one of the $m$ inserted points to zero in $R_{N,1}$. Its expression is then an expectation under $\CJ_{N,\beta,\alpha/2}$ of normalized-polynomial powers of total exponent $(m-1)\beta<\alpha$. The same coupling and moment argument, or \cite[Proposition 2.14]{AN}, gives
\[
 R_{N,1}(t)\longrightarrow R_{\beta,m}(t;b).
\]
Section~\ref{sec:normalization} shows $\kappa_N\to\kappa_{\beta,m}$, proving \eqref{eq:identity}.

At criticality, choose $1<q<1+\beta$. Lemma~\ref{lem:compact} gives
\[
 \E\sup_{|z|\leq r_1}|\zeta'(z)|<\infty.
\]
For nonnegative $v_j$ and weights $r_j\geq0$ with $\sum_jr_j=1$, their weighted geometric mean lies between $\min_jv_j$ and $\max_jv_j$. Thus
\[
 \left|\prod_j|\zeta(z_j)|^{r_j}-1\right|
 \leq\max_j|\zeta(z_j)-1|
 \leq\max_j|z_j|\sup_{|z|\leq r_1}|\zeta'(z)|.
\]
Use $z_j=-x,tb_i-x$ and the exponents in \eqref{eq:K} to obtain \eqref{eq:Klipschitz}. This completes the proof of Lemma~\ref{lem:interpolation}.

\section{The subcritical asymptotic}\label{sec:subcritical}
Assume $0<\alpha<1$. For $P(u)=\prod_i|u-b_i|^\beta$, centering gives
\begin{equation}\label{eq:tailP}
 P(u)=|u|^\alpha\left(1-\frac{\beta B}{2u^2}+O(|u|^{-3})\right),
 \qquad |u|\to\infty.
\end{equation}
Hence $P(u)-|u|^\alpha$ is absolutely integrable. For almost every $u$,
\begin{align}
 \int_0^1 |u|^{\alpha(1-s)}P(u)^s
       \log\frac{P(u)}{|u|^\alpha}\dd s&=P(u)-|u|^\alpha,\label{eq:signedint}\\
 \int_0^1 |u|^{\alpha(1-s)}P(u)^s
       \left|\log\frac{P(u)}{|u|^\alpha}\right|\dd s&=|P(u)-|u|^\alpha|.\label{eq:absoluteint}
\end{align}
The second identity holds because the logarithm has one sign for each fixed $u$.

In \eqref{eq:identity}, set $x=tu$. The factor outside the integral is $t^{\alpha+1}$. For fixed $u$, $K_{s,t}(tu)\to1$, uniformly in $s$. The bound $K\leq1$ and \eqref{eq:absoluteint} give an integrable majorant. Therefore
\begin{equation}\label{eq:rawcoefficient}
 \lim_{t\downarrow0}\frac{R_{\beta,m}(t;b)-1}{t^{\alpha+1}}
 =\kappa_{\beta,m}\int_\R\{P(u)-|u|^\alpha\}\dd u.
\end{equation}

To identify its sign, put $A(u)=m^{-1}\sum_i|u-b_i|^\alpha$. By centering, $A(u)-|u|^\alpha$ is also absolutely integrable. For $L>\max|b_i|$, direct integration gives
\[
 \int_{-L}^L (|u-c|^\alpha-|u|^\alpha)\dd u
 =\frac{(L+c)^{\alpha+1}+(L-c)^{\alpha+1}-2L^{\alpha+1}}{\alpha+1}
 =O(L^{\alpha-1}).
\]
Average this equality over $c=b_i$, then let $L\to\infty$. It follows that
\[
 \int_\R(P(u)-|u|^\alpha)\dd u
 =-\int_\R(A(u)-P(u))\dd u=-\mathcal J_{\beta,m}(b).
\]
The arithmetic--geometric mean inequality gives $A\geq P$, because $\alpha/m=\beta$. Distinct profile entries imply strict inequality on a set of positive measure. Moreover,
\[
 A(u)-P(u)=\frac{\beta^2 V(a)}2|u|^{\alpha-2}
                         +O(|u|^{\alpha-3}),
\]
so the deficit is integrable and has a strictly positive integral. This proves \eqref{eq:sub}. Translation invariance follows by a change of integration variable in \eqref{eq:J}, and scaling gives its asserted homogeneity.

For local uniformity, on bounded centered profiles the tail bound in \eqref{eq:tailP} is uniform. On bounded $u$ intervals, use \eqref{eq:absoluteint} and the uniform convergence of $K$ near the origin. These control the integrated error uniformly, proving the stated remainder claim.

\section{Criticality and the logarithmic coefficient}\label{sec:critical}
Let $\alpha=1$, so $\beta=1/m$. Fix $r\in(0,r_0/2)$ and take $t$ sufficiently small. By \eqref{eq:Klipschitz} and \eqref{eq:absoluteint},
\begin{align*}
 &\left|\int_{|x|<r}\int_0^1 W_{s,t}(x)L_t(x)(K_{s,t}(x)-1)\dd s\dd x\right|\\
 &\hspace{25mm}\leq C\int_{|x|<r}(|x|+t)|P_t(x)-|x||\dd x=O(t^2).
\end{align*}
To verify the last bound, use $|P_t(x)-|x||\leq Ct$ for $|x|\leq2Mt$, and $|P_t(x)-|x||\leq Ct^2/|x|$ for $2Mt<|x|<r$. On $|x|\geq r$, the tail estimates and the decay of $K$ give an $O(t^2)$ contribution. Therefore
\begin{equation}\label{eq:criticallocal}
 R_{1/m,m}(t;b)-1
 =\kappa_{1/m,m}\int_{|x|<r}(P_t(x)-|x|)\dd x+O(t^2).
\end{equation}
For $2Mt<|x|<r$, \eqref{eq:tailP} yields
\[
 P_t(x)-|x|=-\frac{\beta B}{2}\frac{t^2}{|x|}
                           +O\left(\frac{t^3}{x^2}\right).
\]
Integrate over both signs of $x$. The region $|x|\leq2Mt$ contributes $O(t^2)$, and hence
\[
 \int_{|x|<r}(P_t(x)-|x|)\dd x
 =-\beta B\,t^2\log(1/t)+O(t^2).
\]
Using \eqref{eq:kappa},
\[
 \kappa_{1/m,m}=\frac1{8(2m+1)},\qquad
 \beta B=\frac{V(a)}{m^2},
\]
which proves \eqref{eq:crit}. All constants above can be chosen uniformly on bounded centered profiles. Notice that the $O(t^2)$ remainder uses the proved derivative bound \eqref{eq:Klipschitz}; continuity alone would only give an $o(t^2\log(1/t))$ error.

\section{Normalization constants}\label{sec:normalization}
Put $h=\beta/2$ and $g_j=\Gamma(1+jh)$. The circular Selberg--Morris normalization \cite{Forrester} is
\begin{equation}\label{eq:Morris}
 Z_N(0)=\frac{\Gamma(1+Nh)}{\Gamma(1+h)^N},\qquad
 \frac{Z_N(q)}{Z_N(0)}=
 \prod_{j=0}^{N-1}\frac{\Gamma(1+jh)\Gamma(1+q+jh)}
                         {\Gamma(1+q/2+jh)^2}.
\end{equation}
At $q=2mh$, the second expression telescopes:
\[
 M_N(2mh):=\frac{Z_N(2mh)}{Z_N(0)}
 =\frac{\prod_{j=N+m}^{N+2m-1}g_j}{\prod_{j=N}^{N+m-1}g_j}
    \frac{\prod_{j=0}^{m-1}g_j}{\prod_{j=m}^{2m-1}g_j}.
\]
Substitution in \eqref{eq:kappaNdef} gives the exact finite-$N$ value
\begin{equation}\label{eq:kappaN}
 \kappa_N=\frac1{2\pi N^\alpha}
 \frac{\Gamma(1+h)\Gamma(1+(N+2m)h)}{\Gamma(1+Nh)}
 \frac{\Gamma(1+mh)^2}{\Gamma(1+2mh)\Gamma(1+(2m+1)h)}.
\end{equation}
The gamma-ratio asymptotic proves \eqref{eq:kappa}. Alternatively, the leading correlation constant can itself be written without the auxiliary Barnes-type function:
\begin{equation}\label{eq:Cgamma}
 C_\beta^{(m)}=
 \frac{h^{h m(m-1)}\Gamma(1+h)^m}{(2\pi)^m}
       \frac{\prod_{j=0}^{m-1}\Gamma(1+jh)}
            {\prod_{j=m}^{2m-1}\Gamma(1+jh)}.
\end{equation}
It follows by taking the fusion limit in the finite circular correlation density and then applying \eqref{eq:Morris}. Dividing the formula for $m+1$ by that for $m$ gives $\kappa_{\beta,m}$ directly. At criticality, \eqref{eq:kappaN} also yields the useful exact check
\[
 \kappa_N=\frac1{8(2m+1)}\left(1+\frac{2m}{N}\right).
\]

\section{Closed evaluation of the pair coefficient}\label{sec:pair}
We now prove Corollary~\ref{cor:pair}. In this section, $\mathrm B$ denotes the Euler beta function.
\begin{proof}
For $m=2$, the deficit is
\[
 \mathcal J_{\beta,2}(0,1)=\frac12\int_\R
                  (|u|^\beta-|u-1|^\beta)^2\dd u.
\]
Split at $0,1$. The middle interval contributes
\[
 \frac1{2\beta+1}-\mathrm B(\beta+1,\beta+1).
\]
The two exterior intervals together contribute
\[
 H=\int_0^\infty((x+1)^\beta-x^\beta)^2\dd x.
\]
For completeness, write $T(R)=\int_0^R x^\beta(1+x)^\beta\dd x$. Substituting $y=x/(1+x)$, and subtracting the first two endpoint divergences, gives
\[
 T(R)=\frac{R^{2\beta+1}}{2\beta+1}+\frac12R^{2\beta}
       +\frac{\Gamma(1+\beta)\Gamma(-1-2\beta)}{\Gamma(-\beta)}+o(1).
\]
This is also the finite-part Euler beta integral, with its two displayed divergent terms; it follows by integrating the first two powers in the endpoint expansion, the residual being integrable. Integrating the two square terms in $H$ explicitly and cancelling their endpoint powers yields
\[
 H=-\frac1{2\beta+1}
       -2\frac{\Gamma(1+\beta)\Gamma(-1-2\beta)}{\Gamma(-\beta)}.
\]
The gamma reflection identity then gives
\[
 \mathcal J_{\beta,2}(0,1)
 =\mathrm B(\beta+1,\beta+1)\bigl(\sec(\pi\beta)-1\bigr).
\]
Equations \eqref{eq:pairsub}--\eqref{eq:paircrit} now follow from Theorem~\ref{thm:main}.
\end{proof}

\section{The inverse-square threshold and further questions}\label{sec:discussion}
The two asymptotic regimes arise from the tail of the same elementary collision kernel. For centered profiles,
\[
 P(u)-|u|^\alpha=-\frac{\beta B}{2}|u|^{\alpha-2}
                          +O(|u|^{\alpha-3}).
\]
When $\alpha<1$, this tail is integrable, and rescaling $x=\varepsilon u$ gives the entire first correction as a finite one-dimensional integral. When $\alpha=1$, the tail is proportional to $1/|u|$; the interval between the collision scale and a fixed microscopic scale contributes the logarithm. The local Lipschitz estimate for the interpolation kernel controls the error at order $\varepsilon^2$.

For comparison, the inverse-square identity in \cite[Theorem 1.3]{Fang} is
\[
 \E_{\HP_{\beta,m\beta/2}}\sum_x\frac1{x^2}
       =\frac{m\beta}{4(m\beta-1)(2m+1)},\qquad m\beta>1.
\]
The critical coefficient obtained here agrees with the residue of the normalized supercritical coefficient:
\begin{equation}\label{eq:residue}
 \lim_{\beta\downarrow1/m}(m\beta-1)
 \frac{\beta^2V(a)}{8(m\beta-1)(2m+1)}
       =\frac{V(a)}{8m^2(2m+1)}.
\end{equation}
The logarithmic law was proved directly in Section~\ref{sec:critical}; the residue identity is a consistency check. In the pair case, the subcritical expression gives a matching check from the other side:
\[
 \lim_{\beta\uparrow1/2}(1-2\beta)D_\beta=\frac1{160}.
\]
Indeed $\mathrm B(3/2,3/2)=\pi/8$, $\kappa_{1/2,2}=1/40$, and
$\sec(\pi\beta)\sim2/[\pi(1-2\beta)]$.

\subsection*{Further questions}
The present results take $\beta$ fixed before sending $\varepsilon$ to zero. A uniform transition formula when $(m\beta-1)\log(1/|\varepsilon|)$ remains bounded would require joint parameter control of the interpolation kernel. The separate asymptotic formulas established here do not justify interchanging these limits.

At criticality, the next task is to identify the coefficient of the $\varepsilon^2$ term after the logarithm has been removed. The interpolation identity separates the explicit local singularity from a finite contribution involving the full stochastic-zeta kernel, which provides a starting point for such a calculation. Higher collision coefficients and corresponding integrability thresholds can also be investigated through further derivatives or particle selections in the finite-dimensional identity.

For general beta log-gases, a quantitative transfer of the new coefficients from the exact $\Sine_\beta$ limit would require universality estimates that remain valid as the correlation arguments merge. Ordinary convergence at fixed distinct arguments alone does not provide the required rate relative to the vanishing Vandermonde factor.

\section*{AI-assisted tools}
OpenAI ChatGPT was used for exploratory derivations, algebraic consistency checks, literature navigation, and manuscript preparation. The author is responsible for the mathematical claims and references.

\end{document}